\documentclass[10pt]{amsart}

\usepackage{amsmath,amssymb,amsthm,mathrsfs,graphicx,mathtools}
\usepackage[colorlinks=true,linkcolor=blue,citecolor=blue]{hyperref}
\newtheorem{theorem}{Theorem}[section]
\newtheorem{lemma}[theorem]{Lemma}
\newtheorem{proposition}[theorem]{Proposition}

\theoremstyle{definition}

\theoremstyle{remark}
\newtheorem{remark}[theorem]{Remark}

\newcommand{\R}{\mathbb{R}}
\newcommand{\Area}{\operatorname{Area}}
\newcommand{\T}{\mathbb{S}^1}
\newcommand{\C}{\mathbb C}
\newcommand{\Z}{\mathbb Z}

\numberwithin{equation}{section}


\hypersetup{
pdftitle={A Fourier approach to Gromov's Filling Area Conjecture},
pdfauthor={Le Chen, Xiaolong Li, Yimin Zhong},
pdfkeywords={filling area conjecture, isometric filling, Fourier calibration, comass, injective hull}
}

\title{A Fourier approach to Gromov's Filling Area Conjecture}

\author{Le Chen}
\address{Department of Mathematics and Statistics, Auburn University, Auburn, AL 36849}
\email{lzc0090@auburn.edu}
\thanks{Le Chen was partially supported by NSF CAREER grant DMS-2443823.}

\author[Xiaolong Li]{Xiaolong Li}
\address{Department of Mathematics and Statistics, Auburn University, Auburn, AL, 36849}
\email{xil0005@auburn.edu}
\thanks{Xiaolong Li was supported in part by NSF-DMS \#2553660 and a start-up grant at Auburn University.}

\author{Yimin Zhong}
\address{Department of Mathematics and Statistics, Auburn University, Auburn, AL 36849}
\email{yzz0225@auburn.edu}
\thanks{Yimin Zhong was partially supported by NSF grant DMS-2309530.}

\begin{document}

\begin{abstract}
We prove that every compact connected Riemannian isometric filling \(M\) of a circle of length \(2\pi\) satisfies \(\operatorname{Area}(M) \geq \frac{14\zeta(3)}{\pi} \approx 5.35677 \), regardless of orientability or topological types. Our new approach uses the odd Fourier coefficients of the distance functions from boundary points. For orientable fillings, we use a cubic resonant perturbation to obtain \(\operatorname{Area}(M)>5.40154\). 
\end{abstract}

\maketitle

\section{Introduction}
A compact, connected Riemannian surface $M$ with one boundary component $\partial M$ is called a \emph{Riemannian isometric filling} of $\T$, a circle of length \(2\pi\), if there exists an arclength parametrization $\gamma:\T\to\partial M$ such that
\begin{equation}\label{eq 1.1}
    d_M(\gamma(s),\gamma(t))=d_{\T}(s,t) \qquad \text{for all }s,t\in\T,
\end{equation}
where $d_M$ and $d_{\T}$ denote the intrinsic distances on $M$ and $\T$, respectively.
Gromov’s filling area conjecture \cite{Gromov83} asserts that every orientable Riemannian isometric filling of $\T$ satisfies
\[
\operatorname{Area}(M)\ge 2\pi.
\]
The conjectured bound is attained by the unit round hemisphere. Gromov proved the conjecture when \(M\) is homeomorphic to a disk \cite[pages 59-60]{Gromov83}, using Pu’s systolic inequality \cite{Pu52}. 
Bangert, Croke, Ivanov, and Katz subsequently established the conjecture for every orientable filling of genus one, as a consequence of their work on ovalless real hyperelliptic surfaces \cite{BCIK05}. Despite the partial results and related works in \cite{Ivanov11}, \cite{Cossarini20}, \cite{Ambi23}, \cite{Zust25} and  \cite{Chambers26}, the conjecture remains open in general for genus at least two. 

In this paper, we develop a Fourier approach to obtain lower bounds for the filling area that are uniform over all topological types. Our main theorem states that
\begin{theorem}\label{thm:main}
Every compact, connected Riemannian isometric filling \(M\) of a circle of length \(2\pi\), orientable or not, satisfies
\begin{equation}\label{eq 1.2}
\operatorname{Area}(M)\ge \frac{14\zeta(3)}{\pi}
\approx 5.35677,
\end{equation}
where \(\zeta\) denotes the Riemann zeta function.    
\end{theorem}

For orientable fillings, we further improve this estimate to
\begin{theorem}\label{thm:main orientable}
Every compact, connected orientable Riemannian isometric filling \(M\) of a circle of length \(2\pi\) satisfies
\[
\operatorname{Area}(M)>5.40154.
\]   
\end{theorem}

The proof of Theorem \ref{thm:main} is based on the distance functions from boundary points. Their odd Fourier coefficients define planar maps whose two-dimensional Jacobians satisfy a common pointwise bound. A key step is to construct finite orthogonal combinations of these maps whose boundary traces are Jordan curves, while preserving the Jacobian bound and the total boundary action. The area formula and the relative fundamental class with coefficients in \(\Z_2\) then show that the areas enclosed by these curves give lower bounds for the integrals of the Jacobians. This proves Theorem \ref{thm:main} without any orientability assumption. For orientable fillings, we introduce a cubic perturbation that couples Fourier modes with the same boundary frequency. This perturbation increases the boundary action to first order, while a quantitative estimate near the tangent planes on which the unperturbed form is extremal bounds the comass of the pulled-back two-form by \(1+O(\lambda^2)\). Stokes' theorem then gives Theorem \ref{thm:main orientable}.

The use of boundary-distance functions in our proof is related to several results in Finsler geometry. Ivanov \cite{Ivanov11} proved that a Finsler metric on the disk whose geodesics are minimizing is a minimal filling of its boundary among disk fillings, where area is understood in the Holmes-Thompson sense. His argument uses the cyclic order of the gradients of distance functions and also gives a Finsler extension of Pu's inequality. More recently, Chambers \cite{Chambers26} obtained a genus-dependent extension of Ivanov's filling comparison theorem. His proof combines special distance functions and Stokes' theorem with an estimate depending on the genus. Our construction also begins with boundary-distance functions. In contrast to these arguments, we use their Fourier coefficients rather than the pointwise cyclic order of their gradients, and the resulting estimate is independent of the topological type of the filling.

A different approach, closer in spirit to our orientable argument, is based on differential forms in the injective hull of the circle. Z\"ust \cite{Zust25} constructed an exact two-form and established an almost-calibration estimate near the hemisphere for oriented fillings that agree with the hemisphere on a fixed boundary collar. His global projection argument also gives
\[
\Area(M)\geq \frac{\pi^2}{2} \approx 4.93480
\]
for every orientable Riemannian isometric filling of \(\T\). The pulled-back two-form used in the proof of Theorem \ref{thm:main orientable} is similar in spirit, but our Fourier perturbation does not require a prescribed boundary collar. In addition, Theorem \ref{thm:main} improves the above estimate to \eqref{eq 1.2} and does not require orientability.

There is also a discrete approach to the filling problem. Cossarini \cite{Cossarini20} introduced a discretization by systems of curves, called wall systems, and proved that the resulting discrete filling conjecture is equivalent to its reversible Finsler counterpart.
This approach gives another proof of the disk case and establishes the corresponding filling inequality for M\"obius bands with Holmes-Thompson area. More recently, Briggs and Wells \cite{BW26} introduced a graph-theoretic relaxation and proved that
\[
\Area(M)\geq \frac{\sqrt{3}}{4}\pi^2 \approx 4.27366
\]
for every compact Riemannian isometric filling of \(\T\), including
nonorientable fillings. Their passage from the discrete estimate to the Riemannian inequality uses approximation by fine triangulations. It remains open whether their discrete problem is equivalent to Gromov's filling area conjecture. Theorem \ref{thm:main} improves their estimate while applying to the same class of fillings.

We leave open the question of improving the bounds by other nonlinear couplings or by using the topology of the filling. 

The paper is organized as follows. In Section \ref{sec:distance and Fourier}, we introduce the boundary-distance functions and their Fourier coefficients, and prove the pointwise Jacobian estimate used throughout the paper. In Section \ref{sec:universal-proof}, we prove the universal bound by combining orthogonal mixing with a planar coverage argument. In Section \ref{sec:orientable-case}, we treat orientable fillings by constructing a resonant cubic perturbation of the Fourier map and comparing the resulting increase in boundary action with the corresponding change in comass.

\subsection*{AI Disclosure} We used OpenAI language models while developing and checking the Fourier and nonlinear-resonance arguments. The authors are responsible for the
mathematical content. 
\subsection*{Formalization} Lean proofs of Theorems~\ref{thm:main}, \ref{thm:main orientable}, and \ref{thm:nonlinear-formula}, using Lean and Mathlib 4.29.0, are available in the \href{https://github.com/lowrank/gromov-filling-lower-bound/tree/ca06f37383f51bc7bbefd30f036820c83573b427}{companion repository}.

\section{Distance Functions and Fourier Coefficients}
\label{sec:distance and Fourier}

Let $\T=\mathbb R/2\pi\mathbb Z$ denote the circle of length $2\pi$, equipped with its intrinsic distance
$$
d_{\T}(s,t)=\min_{k\in\mathbb Z}|s-t+2\pi k|.
$$
Throughout this section, we fix a Riemannian isometric filling $M$ of $\T$ and an arclength parametrization $\gamma: \T \to \partial M$.

\subsection{Distance functions and the antipodal decomposition}
For each $\theta\in\T$, define the distance function $u_\theta: M \to \R$ by 
$$
u_\theta(x):=d_M(x,\gamma(\theta)). 
$$
The triangle inequality gives
\begin{equation}\label{eq 2.1}
    |u_\theta(x)-u_\varphi(y)|
\le d_M(x,y)+d_{\T}(\theta,\varphi)
\end{equation}
for all $x,y\in M$ and $\theta,\varphi\in\T$. Consequently, $u_\theta(x)$ is $1$-Lipschitz in each variable separately and jointly locally Lipschitz in $(x,\theta)$. 
Define
$$a_\theta(x):=\frac{u_\theta(x)-u_{\theta+\pi}(x)}{2},$$
and
$$\sigma_\theta(x):=\frac{u_\theta(x)+u_{\theta+\pi}(x)-\pi}{2}.$$
Then we have 
$$u_\theta-\frac{\pi}{2}=a_\theta+\sigma_\theta, \qquad a_{\theta+\pi}=-a_\theta,
\qquad \sigma_{\theta+\pi}=\sigma_\theta.$$
Thus, $a_\theta$ and $\sigma_\theta$ are the odd and even parts of $u_\theta-\frac{\pi}{2}$ under the antipodal transformation.

Note that both $a_\theta(x)$ and $\sigma_\theta(x)$ satisfy the same joint Lipschitz estimate \eqref{eq 2.1} as $u_\theta$. For every fixed $x\in M$, the function $\theta\mapsto a_\theta(x)$ is $1$-Lipschitz. Hence it is absolutely continuous and satisfies
$$|\partial_\theta a_\theta(x)|\leq1$$
for almost every $\theta\in\T$.

The following is the standard eikonal identity for distance functions;
compare \cite[Theorem 6.31]{Lee18book}. Since $d_M$ is the intrinsic
distance of a manifold with boundary, we include the short argument.
Unless otherwise specified, all gradients and differentials below are
taken with respect to $x$.

\begin{lemma}\label{lemma 2.2}
For every fixed $\theta\in\T$,
$$|\nabla u_\theta(x)|=1$$
at every point $x\in M^\circ$ at which $u_\theta$ is differentiable. Consequently, this identity holds for almost every $(x,\theta)\in M^\circ\times\T$.
\end{lemma}

\begin{proof}
Since $M$ is compact, $(M,d_M)$ is a compact length space and hence admits shortest paths between any two points. Let $x\in M^\circ$ be a point at which $u_\theta$ is differentiable, and let $\eta:[0,L]\to M$ be an arclength-parametrized shortest path from $x=\eta(0)$ to $\gamma(\theta)=\eta(L)$. Because $x$ is an interior point, an initial segment of $\eta$ is a smooth unit-speed Riemannian geodesic. Every subpath of $\eta$ is minimizing, and therefore
$$u_\theta(\eta(h))=L-h=u_\theta(x)-h$$
for all sufficiently small $h>0$. It follows that
$$du_\theta(x)(\dot\eta(0))=\lim_{h \to 0^+}\frac{u_\theta(\eta(h))-u_\theta(x)}{h}=-1.$$
Since $|\dot\eta(0)|=1$, we obtain $|\nabla u_\theta(x)|\geq1$.
The $1$-Lipschitz property of $u_\theta$ gives the reverse inequality, and hence $|\nabla u_\theta(x)|=1$. The almost-everywhere assertion follows from Rademacher's theorem (see \cite[Theorem 3.2]{EvansGariepy2015}) and Fubini's theorem.
\end{proof}

Applying Lemma \ref{lemma 2.2} to $\theta$ and $\theta+\pi$, the
parallelogram identity gives
\begin{equation}\label{eq 2.5}
|\nabla a_\theta|^2+|\nabla\sigma_\theta|^2
=\frac12\left(
|\nabla u_\theta|^2+|\nabla u_{\theta+\pi}|^2
\right)
=1
\end{equation}
for almost every $(x,\theta)\in M^\circ\times\T$. This identity provides the differential estimate needed for the Fourier argument.

\subsection{Fourier coordinates and the Jacobian estimate}

For every integer $n\geq1$, define
$$A_n(x):=\frac1\pi\int_0^{2\pi}a_\theta(x)\cos(n\theta)\,d\theta,$$
and
$$B_n(x):=\frac1\pi\int_0^{2\pi}a_\theta(x)\sin(n\theta)\,d\theta.$$
The uniform Lipschitz bound for $a_\theta$ shows that $A_n$ and $B_n$ are Lipschitz functions on $M$. We write
$$F_n:=(A_n,B_n):M\longrightarrow \R^2,$$
and 
$$z_n:=A_n+iB_n : M\longrightarrow \C.$$
The antipodal symmetry $a_{\theta+\pi}=-a_\theta$ implies that all even-indexed coefficients vanish. For odd $n$, the Fourier coefficients of $a_\theta$ agree with those of $u_\theta$.

We observe that the boundary map $z_n|_{\partial M}:\partial M \to \C$ can be calculated explicitly.

\begin{lemma}\label{lem:boundary-fourier}
For every positive odd integer $n$,
$$z_n(\gamma(s))=-\rho_ne^{ins},$$
where $\rho_n:=\frac4{\pi n^2}$. 
\end{lemma}

\begin{proof}
On $\partial M$, the isometric filling condition \eqref{eq 1.1} implies
$$u_\theta(\gamma(s))=d_{\T}(\theta,s).$$
It follows from $d_{\T}(\theta,s)+d_{\T}(\theta+\pi,s)=\pi$ that 
$$a_\theta(\gamma(s))
=d_{\T}(\theta,s)-\frac{\pi}{2}=f(\theta-s),$$
where
$$f(t):=d_{\T}(t,0)-\frac{\pi}{2}.$$
Translation invariance on $\T$ gives
$$ z_n(\gamma(s))=\frac{e^{ins}}{\pi}\int_0^{2\pi}f(t)e^{int}\,dt.$$
Identifying $\T$ with $[-\pi,\pi]$, the function $f$ is even and $f(t)=|t|-\pi/2$. Therefore,
\begin{align*}
z_n(\gamma(s))
&=\frac{2e^{ins}}{\pi}
\int_0^\pi\left(t-\frac{\pi}{2}\right)\cos(nt)\,dt\\
&=\frac{2}{\pi n^2}\bigl((-1)^n-1\bigr)e^{ins}.
\end{align*}
For odd $n$, this is $-4e^{ins}/(\pi n^2)$.
\end{proof}

The following elementary lemma is a consequence of Rademacher's theorem, Fubini's theorem, and dominated convergence. We include the details because we need a common set of full measure for all Fourier modes.

\begin{lemma}\label{lemma 2.4}
There exists a set $E\subset M^\circ$ of full measure such that all the functions $A_n$ and $B_n$, $n\geq1$, are differentiable at every $x\in E$ and, for every $v\in T_xM$,
$$dA_n(x)(v)=\frac1\pi\int_0^{2\pi} d a_\theta(x)(v)\cos(n\theta)\,d\theta,$$
and
$$dB_n(x)(v)=\frac1\pi\int_0^{2\pi} d a_\theta(x)(v)\sin(n\theta)\,d\theta.
$$
Moreover, $E$ may be chosen so that \eqref{eq 2.5} holds for almost every $\theta$ at every $x\in E$.
\end{lemma}

\begin{proof}
The map $(x,\theta)\longmapsto a_\theta(x)$ is jointly locally Lipschitz. Rademacher's theorem and Fubini's theorem therefore give a set $E_0\subset M^\circ$ of full measure such that, for every $x\in E_0$, the function $a_\theta$ is differentiable with respect to $x$ for almost every $\theta \in \T$.

Fix $x\in E_0$ and work in local coordinates. Let $L$ be a common local Lipschitz constant for the functions $a_\theta$. For almost every $\theta$, set
$$R_\theta(h):=\frac{a_\theta(x+h)-a_\theta(x)-D_xa_\theta(x)h}{|h|}.$$
Then $R_\theta(h)\to0$ as $h\to0$, while
$$|R_\theta(h)|\leq2L$$
for all sufficiently small $h$. We define $D_xa_\theta(x)=0$ on the exceptional set of $\theta$'s. Its coordinate components are measurable functions of $\theta$, since they are almost-everywhere limits of measurable difference quotients.

For $n\geq1$, define the linear forms
$$\Lambda^A_{n,x}(v):=\frac1\pi\int_0^{2\pi}d a_\theta(x)(v)\cos(n\theta)\,d\theta$$
and
$$\Lambda^B_{n,x}(v):=\frac1\pi\int_0^{2\pi} d a_\theta(x)(v)\sin(n\theta)\,d\theta.$$
Dominated convergence gives
$$\frac{|A_n(x+h)-A_n(x)-\Lambda^A_{n,x}(h)|}{|h|} \longrightarrow 0
$$
as $h\to0$, and the same argument applies to $B_n$. Thus $A_n$ and $B_n$ are differentiable at $x$, with differentials $\Lambda^A_{n,x}$ and $\Lambda^B_{n,x}$, respectively. Since $E_0$ is independent of $n$, these conclusions hold simultaneously for all $n\geq1$.

Applying the preceding argument in a countable coordinate cover and intersecting $E_0$ with the full-measure set obtained from \eqref{eq 2.5} by Fubini's theorem gives the required set $E$.
\end{proof}

We now record the pointwise Jacobian estimate for finite orthogonal combinations of the Fourier coordinate maps.

\begin{proposition}\label{prop 2.5}
Let $m_1,\ldots,m_N$ be distinct positive odd integers, and let $U=(u_{jk})\in O(N)$. Define
$$G_j:=\sum_{k=1}^N u_{jk}F_{m_k},\qquad 1\leq j\leq N.$$
Then
\begin{equation}\label{eq:jacobian-bound}
\sum_{j=1}^N J_2G_j(x)\leq1
\end{equation}
for almost every $x\in M^\circ$. 
Here $J_2G_j$ denotes the two-dimensional
Jacobian.
\end{proposition}

\begin{proof}
Let $E\subset M^\circ$ be the full-measure set provided by Lemma \ref{lemma 2.4}. Fix $x\in E$, and let $(e_1,e_2)$ be an orthonormal basis of $T_xM$. All differentials below are evaluated at $x$.

For $\alpha=1,2$, set
$$h_\alpha(\theta):=da_\theta(e_\alpha).$$
By Lemma \ref{lemma 2.4}, we have
$$dA_{m_k}(e_\alpha)=\frac1{\sqrt\pi}\left\langle
h_\alpha,\frac{\cos(m_k\theta)}{\sqrt\pi}\right\rangle_{L^2},$$
and the analogous identity holds for $dB_{m_k}(e_\alpha)$ with $\sin(m_k\theta)$ in place of $\cos(m_k\theta)$. Since
$$\left\{\frac{\cos(m_k\theta)}{\sqrt\pi},\frac{\sin(m_k\theta)}{\sqrt\pi}
\right\}_{k=1}^N$$
is an orthonormal family in $L^2([0,2\pi])$, Bessel's inequality gives
$$\sum_{k=1}^N|dF_{m_k}(e_\alpha)|^2 \leq\frac1\pi\int_0^{2\pi}
|da_\theta(e_\alpha)|^2\,d\theta.$$
Moreover, $dG_j=\sum_k u_{jk}dF_{m_k}$, so the orthogonality of $U$ implies
$$\sum_{j=1}^N|dG_j(e_\alpha)|^2=\sum_{k=1}^N|dF_{m_k}(e_\alpha)|^2.$$
Consequently,
\begin{align*}
\sum_{j=1}^N J_2G_j(x)
&=\sum_{j=1}^N |dG_j(e_1)\wedge dG_j(e_2)|\\
&\leq\frac12\sum_{j=1}^N \left(|dG_j(e_1)|^2+|dG_j(e_2)|^2\right)\\
&\leq\frac1{2\pi}\int_0^{2\pi} \left(|da_\theta(e_1)|^2+|da_\theta(e_2)|^2 \right)\,d\theta\\
&=\frac1{2\pi}\int_0^{2\pi}|\nabla a_\theta(x)|^2\,d\theta \\
&=1-\frac1{2\pi}\int_0^{2\pi}|\nabla\sigma_\theta(x)|^2\,d\theta \\
& \leq 1,
\end{align*}
where the last equality follows from \eqref{eq 2.5}. This proves \eqref{eq:jacobian-bound} on $E$, and hence almost everywhere in $M^\circ$. 
\end{proof}

The estimate in Proposition \ref{prop 2.5} is uniform in the choice of the frequencies and the orthogonal matrix. In the next section, we use this freedom to choose $m_k$ and $U$ so that every boundary trace $G_j|_{\partial M}$ is a Jordan curve. The areas enclosed by these curves will then give lower bounds for the integrals of the Jacobians.

\section{The Universal Area Bound}\label{sec:universal-proof}

In this section, we prove Theorem \ref{thm:main}. We first construct orthogonal combinations of the odd Fourier modes for which the first boundary mode strictly dominates all the remaining modes. This makes every boundary trace a Jordan curve. We then use the relative fundamental class with coefficients in $\Z_2$ and the area formula to bound the integral of each Jacobian from below. Together with Proposition \ref{prop 2.5}, this gives the desired estimate without any orientability assumption.

Fix $N\geq1$ and set
$$m_k=2k-1,\qquad w_k=\frac1{m_k},\qquad1\leq k\leq N.$$
For $2\leq k\leq N$, let
$$c_k=\frac1{\sqrt{1+4w_k^2}},\qquad s_k=\frac{2w_k}{\sqrt{1+4w_k^2}}.$$
Let $R_k\in O(N)$ act as the identity on the orthogonal complement of
$\operatorname{span}\{e_1,e_k\}$ and have matrix
$$\begin{pmatrix}
c_k&-s_k\\
s_k&c_k
\end{pmatrix}$$
in the ordered basis $(e_1,e_k)$. Define
$$U_N=R_2R_3\cdots R_N=(u_{jk})_{1\leq j,k\leq N}\in O(N).$$
For $N=1$, we take $U_1=(1)$. All empty products below are understood to be $1$.

\begin{lemma}\label{lemma 3.1}
The first row of $U_N$ is given by
$$ u_{11}=\prod_{\ell=2}^N c_\ell,
\qquad
u_{1k}=-s_k\prod_{\ell=2}^{k-1}c_\ell
\quad (2\leq k\leq N).
$$
For $2\leq j\leq N$, we have
$$u_{j1}=s_j\prod_{\ell=j+1}^N c_\ell,
\qquad
u_{jk}=
\begin{cases}
0,&2\leq k<j,\\
c_j,&k=j,\\
-s_js_k\displaystyle\prod_{\ell=j+1}^{k-1}c_\ell,
&j<k\leq N.
\end{cases}
$$
\end{lemma}

\begin{proof}
For $k\geq2$, let $P_{k-1}=R_2\cdots R_{k-1}$, where $P_1=I$, and let $C_1^{(k-1)}$ denote its first column. Since $P_{k-1}$ fixes $e_k$, its $k$-th column is $e_k$. Right multiplication by $R_k$ replaces the first and $k$-th columns by
$$C_1^{(k)}=c_kC_1^{(k-1)}+s_ke_k$$
and
$$C_k^{(k)}=-s_kC_1^{(k-1)}+c_ke_k,$$
respectively, and leaves all other columns unchanged. The first recursion determines the first column of $U_N$. The second determines the $k$-th column, which is unchanged by all subsequent rotations. The stated formulas follow.
\end{proof}

\begin{lemma}\label{lemma 3.2}
Every row of $U_N$ satisfies
$$|u_{j1}|>\sum_{k=2}^N\frac{|u_{jk}|}{m_k},\qquad 1\leq j\leq N.$$
\end{lemma}

\begin{proof}
For $1\leq r\leq N$, write
$$S_r=\sum_{k=r+1}^Nw_k^2.$$
The midpoint inequality for the strictly convex function $x\mapsto(2x-1)^{-2}$ gives
$$S_r \leq\sum_{k=r+1}^{\infty}\frac1{(2k-1)^2} < \int_{r+1/2}^{\infty}\frac{dx}{(2x-1)^2} =\frac1{4r}. $$
Moreover,
$$\log c_k=-\frac12\log(1+4w_k^2)\geq-2w_k^2.$$
It follows that, for every finite set $I\subset\{2,\ldots,N\}$,
$$\prod_{k\in I}c_k \geq \exp\left(-2\sum_{k\in I}w_k^2\right) \geq1-2\sum_{k\in I}w_k^2. $$
By Lemma \ref{lemma 3.1} and the identity $s_k=2w_kc_k\leq2w_k$, we have
$$|u_{11}|\geq1-2S_1$$
and
$$\sum_{k=2}^Nw_k|u_{1k}| \leq\sum_{k=2}^Nw_ks_k \leq2S_1.$$
Consequently,
$$|u_{11}|-\sum_{k=2}^Nw_k|u_{1k}| \geq1-4S_1>0. $$

For $j\geq2$, Lemma \ref{lemma 3.1} gives
$$|u_{j1}|=2w_jc_j\prod_{\ell=j+1}^Nc_\ell \geq2w_jc_j(1-2S_j), $$
whereas
\begin{align*}
\sum_{k=2}^Nw_k|u_{jk}|
&=w_jc_j
+s_j\sum_{k=j+1}^Nw_ks_k
\prod_{\ell=j+1}^{k-1}c_\ell\\
&\leq w_jc_j(1+4S_j).
\end{align*}
Therefore,
$$|u_{j1}|-\sum_{k=2}^Nw_k|u_{jk}| \geq w_jc_j(1-8S_j)>0,$$
where the last inequality follows from $S_j<1/(4j)\leq1/8$.
\end{proof}

The preceding domination inequality is chosen precisely for the
following elementary lemma.

\begin{lemma}\label{lemma 3.3}
Let $m_1,\ldots,m_N$ be positive integers such that $1=m_1<m_2<\cdots<m_N,$
and let $a_1,\ldots,a_N\in\C$. Suppose that
$$ |a_1|>\sum_{k=2}^Nm_k|a_k|.$$
Then the map
$$ z:\T\longrightarrow\C, \qquad z(s)=\sum_{k=1}^Na_ke^{im_ks}, $$
is a smooth embedding with winding number $1$ about the origin. If $\Omega_z$ is the bounded component of $\C\setminus z(\T)$, then
$$ |\Omega_z|=\pi\sum_{k=1}^Nm_k|a_k|^2.$$
\end{lemma}

\begin{proof}
For every positive integer $m$, the factorization of $\xi^m-\eta^m$ gives
$$|e^{ims}-e^{imt}|\leq m|e^{is}-e^{it}|.$$
Thus, for distinct $s,t\in\T$,
$$|z(s)-z(t)| \geq \left(|a_1|-\sum_{k=2}^Nm_k|a_k|\right) |e^{is}-e^{it}|>0.$$
Moreover,
$$|z'(s)| \geq|a_1|-\sum_{k=2}^Nm_k|a_k|>0.$$
It follows that $z$ is a smooth embedding. The homotopy
$$z_t(s)=a_1e^{is}+t\sum_{k=2}^Na_ke^{im_ks},\qquad 0\leq t\leq1, $$
does not meet the origin, since
$$\sum_{k=2}^N|a_k| \leq\sum_{k=2}^Nm_k|a_k|<|a_1|.$$
Hence $z$ has winding number $1$ about the origin. Since $z$ is a Jordan curve, the origin lies in $\Omega_z$ and the parametrization is positively oriented. Green's formula and Fourier orthogonality now give
$$|\Omega_z|=\frac12\operatorname{Im}\int_0^{2\pi} \overline{z(s)}z'(s)\,ds =\pi\sum_{k=1}^Nm_k|a_k|^2.$$
\end{proof}

We also need the following topological fact. The use of coefficients
in $\Z_2$ is what allows us to treat orientable and
nonorientable fillings simultaneously.

\begin{lemma}\label{lemma 3.4}
Let $M$ be a compact connected surface with exactly one boundary component, and let $G:M\to\R^2$ be continuous. Suppose that $G|_{\partial M}$ is an embedding, and let $\Omega$ be the bounded component of $\R^2\setminus G(\partial M)$. Then
$$\Omega\subset G(M).$$
If $M$ is Riemannian and $G$ is Lipschitz, then
$$ \int_MJ_2G\,dA\geq|\Omega|. $$
\end{lemma}

\begin{proof}
Let $i:\partial M\hookrightarrow M$ denote the inclusion. With coefficients in $\Z_2$, the relative fundamental class satisfies
$$\partial_*[M,\partial M]_2=[\partial M]_2;$$
see \cite[Section 3.3, Exercise 31]{Hatcher02}. Exactness of the homology sequence
of the pair gives
$$i_*[\partial M]_2=0
\quad\text{in }H_1(M;\Z_2).$$
Suppose that $y\in\Omega\setminus G(M)$. Then
$$
H_y:M\longrightarrow S^1,
\qquad
H_y(x)=\frac{G(x)-y}{|G(x)-y|},
$$
is well defined. Since $G|_{\partial M}$ parametrizes a Jordan curve
enclosing $y$, the map $H_y|_{\partial M}$ has degree $\pm1$ over
$\mathbb Z$, and hence degree $1$ modulo $2$.
Therefore,
$$(H_y|_{\partial M})_*[\partial M]_2=[S^1]_2\neq0.$$
On the other hand,
$$(H_y|_{\partial M})_*[\partial M]_2=(H_y)_*i_*[\partial M]_2=0,$$
a contradiction. This proves $\Omega\subset G(M)$.

Now suppose that $G$ is Lipschitz. Since $\partial M$ has
two-dimensional measure zero, the area formula (see \cite[Theorem 3.8]{EvansGariepy2015}), applied to a measurable
partition of $M^\circ$ subordinate to coordinate charts, gives
$$\int_MJ_2G\,dA =\int_{\R^2}\#\{x\in M^\circ:G(x)=y\} \,dy. $$
Every $y\in\Omega$ has a preimage in $M^\circ$, because $\Omega\subset G(M)$ and $G(\partial M)=\partial\Omega$. Hence $\#\{x\in M^\circ:G(x)=y\}\geq1$ on $\Omega$, and the desired inequality follows.
\end{proof}

We now apply the preceding lemmas to the Fourier maps from Section 2.

\begin{proposition}\label{prop:finite-universal-bound}
Let $M$ be a compact connected Riemannian isometric filling of the circle of length $2\pi$. Then, for every $N\geq1$,
$$\operatorname{Area}(M) \geq \frac{16}{\pi}\sum_{k=1}^N\frac1{(2k-1)^3}.$$
\end{proposition}

\begin{proof}
Let $m_k=2k-1$ and let $U_N=(u_{jk})$ be the orthogonal matrix constructed above. For $1\leq j\leq N$, define
$$G_j=\sum_{k=1}^Nu_{jk}F_{m_k}.$$
These maps are Lipschitz. Identifying $\R^2$ with $\C$, Lemma \ref{lem:boundary-fourier} gives
$$G_j(\gamma(s))=-\sum_{k=1}^Nu_{jk}\rho_{m_k}e^{im_ks},$$
where $\rho_m=\frac4{\pi m^2}$. Since $m_k\rho_{m_k}=\rho_1/m_k$, Lemma \ref{lemma 3.2} implies
$$\sum_{k=2}^Nm_k|u_{jk}|\rho_{m_k}=\rho_1\sum_{k=2}^N\frac{|u_{jk}|}{m_k}<|u_{j1}|\rho_1.$$
Thus Lemma \ref{lemma 3.3}, applied with $a_k=-u_{jk}\rho_{m_k}$, shows that $G_j|_{\partial M}$ is a Jordan curve. Let $\Omega_j$ denote its bounded complementary component. Lemmas \ref{lemma 3.3} and \ref{lemma 3.4} give
$$\int_MJ_2G_j\,dA \geq|\Omega_j| =\pi\sum_{k=1}^Nm_k|u_{jk}|^2\rho_{m_k}^2.$$
Summing over $j$, and using Proposition \ref{prop 2.5} and the orthogonality of $U_N$, we obtain
\begin{align*}
\operatorname{Area}(M)
&\geq\sum_{j=1}^N\int_MJ_2G_j\,dA\\
&\geq\pi\sum_{j=1}^N\sum_{k=1}^N
m_k|u_{jk}|^2\rho_{m_k}^2\\
&=\pi\sum_{k=1}^Nm_k\rho_{m_k}^2\\
&=\frac{16}{\pi}\sum_{k=1}^N\frac1{(2k-1)^3}.
\end{align*}
\end{proof}

We are ready to prove Theorem \ref{thm:main}. 
\begin{proof}[Proof of Theorem \ref{thm:main}]
Letting $N\to\infty$ in Proposition
\ref{prop:finite-universal-bound} and using
$$\sum_{k=1}^{\infty}\frac1{(2k-1)^3}=\left(1-\frac1{2^3}\right)\zeta(3)
=\frac78\zeta(3),$$
we obtain
$$\operatorname{Area}(M)\geq\frac{14\zeta(3)}{\pi}.$$
\end{proof}

\section{The Nonlinear Bound for Orientable Fillings}\label{sec:orientable-case}

Throughout this section, we assume that $M$ is orientable. We choose an orientation on $M$ so that the arclength parametrization $\gamma:\T\to\partial M$ agrees with the induced boundary orientation. We retain the distance functions and Fourier coefficients introduced in Section 2. Our purpose is to improve the universal estimate by a cubic perturbation of the Fourier map.

Set
\begin{equation}\label{eq:nonlinear-constants}
\begin{gathered}
C_0=\frac{14\zeta(3)}{\pi},\qquad
b=\frac{512}{\pi^3}\left(\frac34-\frac{\pi^2}{16}\right),\\
d=\frac{4096}{\pi^5}\left(
\frac78\zeta(3)+1-\frac{\pi^4}{48}\right),\\
C_*=\sqrt2\left[
\frac{16}{\pi^2}
+\frac8{5\pi}\sqrt{2-\frac{16}{\pi^2}}
\right],\\
D_*=\frac{32}{\pi^2},
\qquad
Q_*=\frac{1280}{9\pi^4}+\frac{128}{9\pi^2}.
\end{gathered}
\end{equation}

The following quantitative estimate implies Theorem
\ref{thm:main orientable}.

\begin{theorem}\label{thm:nonlinear-formula}
Let $M$ be a compact connected orientable Riemannian isometric filling of $\T$. For every $0\leq\lambda<\pi^2/32$, we have
\begin{equation}\label{eq:nonlinear-master}
\operatorname{Area}(M)\geq
\frac{C_0+b\lambda+d\lambda^2}
{1+\lambda^2\left[
Q_*+\dfrac{C_*^2}{4(1-D_*\lambda)}
\right]}.
\end{equation}
In particular,
$$\operatorname{Area}(M)>5.40154.$$
\end{theorem}

The numerator in \eqref{eq:nonlinear-master} is the boundary action of the perturbed Fourier map, while the denominator is an upper bound for the comass of its pullback form. The key observation is that the first variation of this form vanishes on every tangent plane where the unperturbed form has comass one. We first obtain a quantitative version of this observation and then combine it with the boundary action.

\subsection{The Fourier map and the cubic perturbation}

Let $\mathcal O=\{1,3,5,\ldots\}$ and let $H=\ell^2(\mathcal O;\C)$, regarded as a real Hilbert space. Define $\Phi:M\longrightarrow H$ by 
$$\Phi(x)=(z_n(x))_{n\in\mathcal O}.$$
On $H$, consider the standard symplectic form
$$\omega(p,q)=\operatorname{Im}\sum_{n\in\mathcal O}\overline{p_n}q_n.$$
Its standard primitive is $\alpha_z(v)=\frac12\omega(z,v),$ so that $d\alpha=\omega$.

\begin{lemma}\label{lem:orientable-fourier-bounds}
The map $\Phi$ is Lipschitz and differentiable almost everywhere. At almost every $x\in M^\circ$, for every orthonormal basis $(e_1,e_2)$ of $T_xM$, we have
\begin{equation}\label{eq:orientable-fourier-energy}
\|d\Phi(e_1)\|_H^2+\|d\Phi(e_2)\|_H^2
=\frac1\pi\int_0^{2\pi}|\nabla a_\theta|^2\,d\theta
\leq2.
\end{equation}
Moreover, writing
$$r=|z_1(x)|,
\qquad
Y^2=\sum_{n\in\mathcal O}|z_{n+2}(x)|^2,
\qquad
Z^2=\sum_{\substack{m\in\mathcal O\\m\geq5}}|z_m(x)|^2,
$$
we have, at every $x\in M$,
\begin{equation}\label{eq:orientable-coordinate-bounds}
r\leq\frac4\pi,
\qquad
r^2+9Y^2\leq2,
\qquad
r^2+25Z^2\leq2.
\end{equation}
\end{lemma}

\begin{proof}
Since $a_\theta(x)-a_\theta(y)$ is real and antipodally odd as a function of $\theta$, Parseval's identity and the Lipschitz estimate for $a_\theta$ give
$$\|\Phi(x)-\Phi(y)\|_H^2=\frac1\pi\int_0^{2\pi}
|a_\theta(x)-a_\theta(y)|^2\,d\theta
\leq2d_M(x,y)^2.$$
Thus $\Phi$ is Lipschitz.

Let $x$ belong to the full-measure set in Lemma \ref{lemma 2.4}, and work in a coordinate chart centered at $x$. Define $\mathcal D_x:T_xM\to H$ by
$$(\mathcal D_xv)_n
=\frac1\pi\int_0^{2\pi}
da_\theta(x)(v)e^{in\theta}\,d\theta,
\qquad n\in\mathcal O.$$
Parseval's identity shows that $\mathcal D_xv\in H$. For a sufficiently small coordinate vector $v$, set
$$R_\theta(v)=a_\theta(x+v)-a_\theta(x)-da_\theta(x)(v).$$
For almost every $\theta$, we have $R_\theta(v)=o(|v|)$ as $v\to0$, whereas $|R_\theta(v)|/|v|$ is uniformly bounded. The function $\theta\mapsto R_\theta(v)$ is again real and antipodally odd. Hence Parseval's identity and dominated convergence give
$$\frac{\|\Phi(x+v)-\Phi(x)-\mathcal D_xv\|_H^2}{|v|^2}
=\frac1\pi\int_0^{2\pi}\frac{|R_\theta(v)|^2}{|v|^2}\,d\theta\longrightarrow0.$$
Therefore, $d\Phi_x=\mathcal D_x$. Another application of Parseval's identity gives
$$\|d\Phi(e_1)\|_H^2+\|d\Phi(e_2)\|_H^2=\frac1\pi\int_0^{2\pi}
\left(|da_\theta(e_1)|^2+|da_\theta(e_2)|^2\right)\,d\theta.$$
The first assertion now follows from \eqref{eq 2.5}.

For every fixed $x\in M$, the function $\theta\mapsto a_\theta(x)$ is absolutely continuous, real, and antipodally odd. Since $|\partial_\theta a_\theta(x)|\leq1$ almost everywhere, the weighted form of Parseval's identity gives
$$\sum_{n\in\mathcal O}n^2|z_n(x)|^2
=\frac1\pi\int_0^{2\pi}
|\partial_\theta a_\theta(x)|^2\,d\theta
\leq2.$$
This proves the last two inequalities in
\eqref{eq:orientable-coordinate-bounds}. Finally, integration by parts gives
$$z_1(x)=-\frac1{i\pi}\int_0^{2\pi}
\partial_\theta a_\theta(x)e^{i\theta}\,d\theta.$$
Choosing $\varphi\in\R$ so that the rotated integral is real and nonnegative, we obtain
$$\pi|z_1(x)|
=\int_0^{2\pi}
\partial_\theta a_\theta(x)
\cos(\theta-\varphi)\,d\theta
\leq\int_0^{2\pi}|\cos(\theta-\varphi)|\,d\theta
=4.$$
This completes the proof.
\end{proof}

Let $S:H\to H$ be the shift defined by
$$
(Sz)_n=z_{n+2},
\qquad n\in\mathcal O.
$$
For $\lambda\geq0$, define
\begin{equation}\label{eq:cubic-fourier-perturbation}
h(z)=\overline{z_1}^{\,2}Sz,
\qquad
G_\lambda(z)=z+\lambda h(z),
\qquad
\Psi_\lambda=G_\lambda\circ\Phi.
\end{equation}
The map $h$ is a smooth real polynomial map, and
\begin{equation}\label{eq:cubic-fourier-derivative}
Dh_z(v)=2\overline{z_1}\,\overline{v_1}\,Sz
+\overline{z_1}^{\,2}Sv.
\end{equation}
In particular,
$$\|Dh_z(v)\|_H\leq3R^2\|v\|_H$$
whenever $\|z\|_H\leq R$. Since $\Phi(M)$ is bounded, this shows that $h$ is Lipschitz on a ball containing $\Phi(M)$, and hence $\Psi_\lambda$ is Lipschitz.

The choice of \eqref{eq:cubic-fourier-perturbation} is motivated by  the boundary identity
$$\overline{z_1(\gamma(s))}^{\,2}z_{n+2}(\gamma(s))=-\rho_1^2\rho_{n+2}e^{ins},$$
where $\rho_n=\frac4{\pi n^2}$. Consequently,
\begin{equation}\label{eq:cubic-fourier-boundary}
(\Psi_\lambda)_n(\gamma(s))
=-\left(\rho_n+\lambda\rho_1^2\rho_{n+2}\right)e^{ins}.
\end{equation}
Define the almost-everywhere pullback form
$$\Omega_\lambda=\Psi_\lambda^*\omega.$$
For this form, the comass means the essential supremum of its absolute value on orthonormal frames in $TM$.

\begin{proposition}\label{prop:nonlinear-boundary-action}
For every $\lambda\geq0$, we have
\begin{equation}\label{eq:nonlinear-boundary-action}
\int_M\Omega_\lambda
=B(\lambda)
:=\pi\sum_{n\in\mathcal O}
n\left(\rho_n+\lambda\rho_1^2\rho_{n+2}\right)^2
=C_0+b\lambda+d\lambda^2.
\end{equation}
In particular, $b>0$.
\end{proposition}

\begin{proof}
Let $P_N$ denote the orthogonal projection onto the span of the first $N$ odd coordinate vectors. The map $P_N\Psi_\lambda$ is finite-dimensional and Lipschitz, and its boundary trace is smooth. The finite-dimensional Stokes formula gives
\begin{align*}
\int_M(P_N\Psi_\lambda)^*\omega
&=\int_{\partial M}(P_N\Psi_\lambda)^*\alpha\\
&=\pi\sum_{\substack{n\in\mathcal O\\n\leq2N-1}}
n\left(\rho_n+\lambda\rho_1^2\rho_{n+2}\right)^2.
\end{align*}
Indeed, the Stokes formula for a Lipschitz map follows by subtracting a smooth extension of its boundary trace and approximating the resulting zero-trace map in $W_0^{1,2}$ by smooth compactly supported maps. The products of the first derivatives then converge in $L^1$.

Choose $R\geq\sup_M\|\Phi\|_H$ and put $K_\lambda=1+3\lambda R^2$. By \eqref{eq:orientable-fourier-energy},
$$\|d\Psi_\lambda(e_1)\|_H^2+\|d\Psi_\lambda(e_2)\|_H^2\leq2K_\lambda^2$$
at almost every point. Therefore,
$$\left|(P_N\Psi_\lambda)^*\omega(e_1,e_2)\right|\leq\frac12\left(\|d\Psi_\lambda(e_1)\|_H^2
+\|d\Psi_\lambda(e_2)\|_H^2\right)\leq K_\lambda^2.$$
Moreover, $P_Nd\Psi_\lambda(e_\alpha)\to d\Psi_\lambda(e_\alpha)$ in $H$ for $\alpha=1,2$ at almost every point. Dominated convergence thus applies to the interior integrals. Since
$$\rho_n+\lambda\rho_1^2\rho_{n+2}=O(n^{-2}),$$
the summands in the boundary action are $O(n^{-3})$. Passing to the limit proves the first equality in \eqref{eq:nonlinear-boundary-action}.

Expanding the square, we obtain
\begin{align*}
B(\lambda)
={}&\pi\sum_{n\in\mathcal O}n\rho_n^2
+2\pi\lambda\rho_1^2
\sum_{n\in\mathcal O}n\rho_n\rho_{n+2}\\
&+\pi\lambda^2\rho_1^4
\sum_{n\in\mathcal O}n\rho_{n+2}^2.
\end{align*}
The constant term is
$$\frac{16}{\pi}\sum_{n\in\mathcal O}\frac1{n^3}=\frac{14\zeta(3)}\pi=C_0.$$
The remaining coefficients follow from
$$\sum_{n\in\mathcal O}\frac1{n(n+2)^2}=\frac34-\frac{\pi^2}{16}$$
and
$$\sum_{n\in\mathcal O}\frac{n}{(n+2)^4}=\frac78\zeta(3)+1-\frac{\pi^4}{48}.$$
For the first identity, we use
$$\frac1{n(n+2)^2}=\frac1{4n}-\frac1{4(n+2)}-\frac1{2(n+2)^2}.$$
For the second identity, substitute $m=n+2$ and write $n/(n+2)^4=m^{-3}-2m^{-4}$. This proves \eqref{eq:nonlinear-boundary-action}. Positivity of $b$ also follows
directly from its defining series.
\end{proof}

\subsection{The first variation}

Let $x\in M^\circ$ belong to the common full-measure set on which the conclusions of Lemmas \ref{lemma 2.4} and \ref{lem:orientable-fourier-bounds} hold. Fix an ordered orthonormal basis $(e_1,e_2)$ of $T_xM$ and write
$$W(\theta)=da_\theta(e_1)+i\,da_\theta(e_2)=\sum_{k\in2\Z+1}c_ke^{ik\theta},$$
where
$$c_k=\frac1{2\pi}\int_0^{2\pi} W(\theta)e^{-ik\theta}\,d\theta.$$
The series is understood in $L^2$. Since $|W|=|\nabla a_\theta|\leq1$ almost everywhere, Parseval's identity gives
\begin{equation}\label{eq:tangent-fourier-energy}
\sum_{k\in2\Z+1}|c_k|^2=\frac1{2\pi}\int_0^{2\pi}|W(\theta)|^2\,d\theta \leq1.
\end{equation}
Set
$$p_n=dz_n(e_1), \qquad q_n=dz_n(e_2), \qquad n\in\mathcal O.$$
The Fourier convention in Section 2 gives
\begin{equation}\label{eq:tangent-complex-coordinates}
p_n=c_{-n}+\overline{c_n},
\qquad
q_n=i\left(\overline{c_n}-c_{-n}\right).
\end{equation}
Consequently,
\begin{equation}\label{eq:unperturbed-density}
\Omega_0(e_1,e_2)=\sum_{n\in\mathcal O}\left(|c_n|^2-|c_{-n}|^2\right).
\end{equation}
In particular, $\|\Omega_0\|_{\mathrm{comass}}\leq1$. Together with Proposition \ref{prop:nonlinear-boundary-action} at $\lambda=0$, this also recovers the universal estimate for orientable fillings.

Since $G_\lambda=\operatorname{Id}+\lambda h$, we have
\begin{equation}\label{eq:nonlinear-form-expansion}
\Omega_\lambda=\Omega_0+\lambda L+\lambda^2Q.
\end{equation}
More precisely, if $z=\Phi(x)$, $p=d\Phi(e_1)$, and $q=d\Phi(e_2)$, then
$$L=\omega(Dh_z(p),q)+\omega(p,Dh_z(q)), \qquad Q=\omega(Dh_z(p),Dh_z(q)).$$
When the frame is suppressed, $\Omega_0$, $L$, and $Q$ denote their evaluations on $(e_1,e_2)$.

\begin{lemma}\label{lem:nonlinear-first-variation}
Writing $a=z_1$ and $y_n=z_{n+2}$, we have
\begin{equation}\label{eq:nonlinear-first-variation}
\begin{aligned}
L={}&2\operatorname{Re}\left[a^2\sum_{n\in\mathcal O}\left(\overline{c_n}c_{n+2}-c_{-n}\overline{c_{-(n+2)}}\right)\right]\\
&+4\operatorname{Re}\left[a\sum_{n\in\mathcal O}\overline{y_n}
\left(c_{-1}\overline{c_n}-\overline{c_1}c_{-n}\right)\right].
\end{aligned}
\end{equation}
In particular, $L=0$ whenever $|\Omega_0|=1$.
\end{lemma}

\begin{proof}
In the following calculation, $dz_n$ and $dh_n$ denote the differentials of the coordinate functions of $\Phi$ and $h\circ\Phi$, respectively. We have
$$L=\frac1{2i}\sum_{n\in\mathcal O}
\left(d\overline{h_n}\wedge dz_n+d\overline{z_n}\wedge dh_n\right)(e_1,e_2),$$
where
$$
\begin{aligned}
dh_n&=2\overline a\,y_n\,d\overline{z_1}+\overline a^{\,2}dz_{n+2},\\ d\overline{h_n}
&=2a\,\overline{y_n}\,dz_1+a^2d\overline{z_{n+2}}.
\end{aligned}
$$
By \eqref{eq:tangent-complex-coordinates},
$$
\begin{aligned}
\frac1{2i}(d\overline{z_j}\wedge dz_k)(e_1,e_2)
&=c_j\overline{c_k}-\overline{c_{-j}}c_{-k},\\ \frac1{2i}(dz_j\wedge dz_k)(e_1,e_2)
&=c_{-j}\overline{c_k}-\overline{c_j}c_{-k}.
\end{aligned}
$$
Substitution and combination of conjugate terms prove \eqref{eq:nonlinear-first-variation}. All the series converge absolutely by Cauchy--Schwarz, since the sequences involved belong to $\ell^2$.

Suppose first that $\Omega_0=1$. Set
$$P=\sum_{n\in\mathcal O}|c_n|^2, \qquad N=\sum_{n\in\mathcal O}|c_{-n}|^2.$$
Equations \eqref{eq:tangent-fourier-energy} and \eqref{eq:unperturbed-density} imply that $P=1$ and $N=0$. Thus $c_{-n}=0$ for every $n\in\mathcal O$, and
$$\frac1{2\pi}\int_0^{2\pi}|W(\theta)|^2\,d\theta=1.$$
Since $|W|\leq1$, we have $|W|=1$ almost everywhere. The second Fourier coefficient of $|W|^2$ therefore vanishes, and hence
$$\sum_{n\in\mathcal O}\overline{c_n}c_{n+2}=0.$$
Both sums in \eqref{eq:nonlinear-first-variation} now vanish, so $L=0$. If $\Omega_0=-1$, we replace $e_2$ by $-e_2$ and apply the preceding argument. Both $\Omega_0$ and $L$ change sign under this replacement, so $L=0$ in this case as well.
\end{proof}

The vanishing in Lemma \ref{lem:nonlinear-first-variation} has the
following quantitative form.

\begin{proposition}\label{prop:nonlinear-linear-estimate}
At every common differentiability point and for every ordered orthonormal basis, set
$$\delta=1-|\Omega_0|.$$
Then
\begin{equation}\label{eq:nonlinear-linear-estimate}
|L|\leq C_*\sqrt\delta+D_*\delta.
\end{equation}
\end{proposition}

\begin{proof}
Reversing the frame if necessary, we may assume that $\Omega_0\geq0$. Write
$$P=\sum_{n\in\mathcal O}|c_n|^2, \qquad N=\sum_{n\in\mathcal O}|c_{-n}|^2. $$
Then $P+N\leq1$, $\Omega_0=P-N$, and
$$ \delta=1-P+N. $$
In particular, $N\leq\delta/2$.

Set
$$A=\sum_{n\in\mathcal O}\overline{c_n}c_{n+2}, \qquad
B=\sum_{n\in\mathcal O} c_{-n}\overline{c_{-(n+2)}}, \qquad
X=\overline{c_{-1}}c_1.$$
The second Fourier coefficient of $|W|^2$ is
$$A+B+X
=\sum_{k\in2\Z+1}\overline{c_k}c_{k+2}=\frac1{2\pi}\int_0^{2\pi}|W(\theta)|^2e^{-2i\theta}\,d\theta.$$
Since $1-|W|^2\geq0$, it follows that
$$|A+B+X|\leq\frac1{2\pi}\int_0^{2\pi}\left(1-|W(\theta)|^2\right)\,d\theta =1-P-N =\delta-2N.$$
Moreover, $|B|\leq N$ and $|X|\leq\sqrt{PN}\leq\sqrt N$. Therefore,
\begin{equation}\label{eq:nonlinear-autocorrelation-bound}
\begin{aligned}
|A-B| &\leq|A+B+X|+2|B|+|X|\\
&\leq\delta+\sqrt N
\leq\delta+\sqrt{\frac\delta2}.
\end{aligned}
\end{equation}

To estimate the second sum in \eqref{eq:nonlinear-first-variation}, set
$$v_n=c_{-1}\overline{c_n}-\overline{c_1}c_{-n}.$$
The first component vanishes, since $v_1=0$. Put $a_+=|c_1|$ and $a_-=|c_{-1}|$. The triangle inequality gives
$$\left(\sum_{n\in\mathcal O}|v_n|^2\right)^{1/2} \leq a_-\sqrt{P-a_+^2}+a_+\sqrt{N-a_-^2}
\leq\sqrt{PN}.$$
For completeness, the last inequality follows from
$$
\begin{aligned}
PN &-\left(a_-\sqrt{P-a_+^2}+a_+\sqrt{N-a_-^2}\right)^2\\
&=\left(\sqrt{(P-a_+^2)(N-a_-^2)}-a_+a_-\right)^2\geq0.
\end{aligned}
$$
Since $v_1=0$, the coordinate $y_1=z_3$ does not contribute. We may therefore use $Z$ from Lemma \ref{lem:orientable-fourier-bounds} to obtain
\begin{equation}\label{eq:nonlinear-cancellation-bound}
\begin{aligned}
\left|\sum_{n\in\mathcal O}\overline{y_n}v_n \right|
&\leq Z\sqrt{PN}\\
&\leq\frac Z2\sqrt{\delta(2-\delta)} \leq Z\sqrt{\frac\delta2}.
\end{aligned}
\end{equation}
Here we used
$$ 4PN=(P+N)^2-(P-N)^2 \leq1-(1-\delta)^2 =\delta(2-\delta).$$
Combining \eqref{eq:nonlinear-first-variation}, \eqref{eq:nonlinear-autocorrelation-bound}, and \eqref{eq:nonlinear-cancellation-bound}, we obtain
$$|L|\leq 2r^2\delta +\sqrt2\left(r^2+2rZ\right)\sqrt\delta,$$
where $r=|z_1|$. By \eqref{eq:orientable-coordinate-bounds},
$$2r^2\leq D_*$$
and
$$r^2+2rZ \leq r^2+\frac25r\sqrt{2-r^2}. $$
The function
$$ f(t)=t+\frac25\sqrt{t(2-t)}$$
is increasing on $[0,1+5/\sqrt{29}]$. Since $r^2\leq16/\pi^2<1+5/\sqrt{29}$, we have
$$\sqrt2\left(r^2+2rZ\right) \leq\sqrt2f\left(\frac{16}{\pi^2}\right) =C_*. $$
This proves \eqref{eq:nonlinear-linear-estimate}.
\end{proof}

\subsection{The quadratic term and the comass estimate}

\begin{lemma}\label{lem:nonlinear-quadratic-estimate}
At every common differentiability point and for every ordered orthonormal basis, we have
$$ |Q|\leq Q_*.$$
\end{lemma}

\begin{proof}
Let $z=\Phi(x)$ and retain the notation $r=|z_1|$ and $Y=\|Sz\|_H$. For every $v\in H$, \eqref{eq:cubic-fourier-derivative} gives
$$\|Dh_z(v)\|_H \leq2rY|v_1|+r^2\|Sv\|_H.$$
Since
$$|v_1|^2+\|Sv\|_H^2=\|v\|_H^2,$$
the Cauchy-Schwarz inequality gives
$$\|Dh_z(v)\|_H \leq\sqrt{4r^2Y^2+r^4}\,\|v\|_H.$$
Using \eqref{eq:orientable-coordinate-bounds}, we obtain
$$ 4r^2Y^2+r^4 \leq\frac89r^2+\frac59r^4 \leq\frac{128}{9\pi^2}+\frac{1280}{9\pi^4} =Q_*.$$
If $p=d\Phi(e_1)$ and $q=d\Phi(e_2)$, then
$$ Q=\omega(Dh_z(p),Dh_z(q)).$$
Consequently, \eqref{eq:orientable-fourier-energy} gives
$$
\begin{aligned}
|Q|&\leq Q_*\|p\|_H\|q\|_H\\
&\leq\frac{Q_*}{2}\left(\|p\|_H^2+\|q\|_H^2\right)\leq Q_*.
\end{aligned}
$$
\end{proof}

\begin{proposition}\label{prop:nonlinear-comass}
For $0\leq\lambda<\pi^2/32$, we have
\begin{equation}\label{eq:nonlinear-comass}
\|\Omega_\lambda\|_{\mathrm{comass}}\leq C(\lambda):=1+\lambda^2\left[Q_*+\frac{C_*^2}{4(1-D_*\lambda)}\right].
\end{equation}
\end{proposition}

\begin{proof}
At a common differentiability point, fix an ordered orthonormal basis and set $\delta=1-|\Omega_0|\in[0,1]$. Proposition \ref{prop:nonlinear-linear-estimate} and Lemma \ref{lem:nonlinear-quadratic-estimate} give
$$
\begin{aligned}
|\Omega_\lambda|&\leq1-\delta+\lambda\left(C_*\sqrt\delta+D_*\delta\right)+Q_*\lambda^2\\
&=1-(1-D_*\lambda)\delta+C_*\lambda\sqrt\delta+Q_*\lambda^2.
\end{aligned}
$$
Since $1-D_*\lambda>0$, completing the square gives
$$-(1-D_*\lambda)\delta+C_*\lambda\sqrt\delta\leq\frac{C_*^2\lambda^2}{4(1-D_*\lambda)}.$$
Taking the essential supremum over all orthonormal frames proves
\eqref{eq:nonlinear-comass}.
\end{proof}

\begin{proof}[Proof of Theorem \ref{thm:nonlinear-formula}]
Propositions \ref{prop:nonlinear-boundary-action} and
\ref{prop:nonlinear-comass} give
$$
B(\lambda)
=\int_M\Omega_\lambda
\leq\int_M|\Omega_\lambda|\,dA
\leq C(\lambda)\operatorname{Area}(M).
$$
Since $C(\lambda)>0$, this proves \eqref{eq:nonlinear-master}. Taking $\lambda=1/25$ and substituting the constants from \eqref{eq:nonlinear-constants}, a direct interval computation gives
$$
5.401544
<
\frac{B(1/25)}{C(1/25)}
<
5.401545.
$$
Consequently,
$$
\operatorname{Area}(M)>5.40154.
$$
This proves Theorem \ref{thm:main orientable}.
\end{proof}

\begin{remark}
The main purpose of this section is to show that a nonlinear resonant
perturbation gives a strict improvement of the universal area bound for
orientable fillings. We have favored explicit estimates over optimizing
all the constants, and the resulting bound is not expected to be sharp.
It remains open whether additional resonant perturbations can produce a
substantially stronger estimate or even approach the conjectured bound
$2\pi$.
\end{remark}

\section*{Conflict of Interest}
The authors state that there is no conflict of interest.

\section*{Data Availability Statement}
Data sharing is not applicable to this article as no datasets were generated or analyzed during the current study.

\bibliographystyle{alpha}
\bibliography{references}

\end{document}